\documentclass{amsart}
\usepackage[T1]{fontenc}
\usepackage{graphicx} 
\usepackage[autostyle,italian=guillemets]{csquotes}
\usepackage{ amssymb }
\usepackage{ mathrsfs }
\usepackage{amsmath}
\usepackage{dsfont}
\usepackage{amsthm}
\usepackage{amssymb}
\usepackage{mathrsfs}
\usepackage{mathtools}
\usepackage{ wasysym }
\usepackage{tikz-cd}
\usepackage{bussproofs}
\newtheorem{teo}{Theorem}[section]

\newtheorem{claim}[teo]{Claim}

\newtheorem{prop}[teo]{Proposition}
\theoremstyle{remark}

\theoremstyle{definition}
\newtheorem{definition}[teo]{Definition}
\DeclareMathOperator{\tp}{tp}
\usepackage{braket}

\title{Hyperdefinability of the Lie Model for approximate subgroups}
\author{Beatrice Degasperi}
\thanks{
This project was partially funded by the Université Franco Italienne by the Vinci project ZAMD\_UIF\_VINCI2\_23\_01.}

\begin{document}
\maketitle

\begin{abstract} In  \cite[Theorem 4.2]{HR2} Hrushovski proves the Lie model theorem in full generality with model theoretic methods. The theorem states that for every approximate group there exists a generalized definable locally compact model, which, simplifying, is a quasi-homomorphism from the group generated by the approximate subgroup to a locally compact group with some particular properties. In  \cite[Theorem 3.25]{KP} Pillay and Krupinski prove the same theorem using topological dynamics on a locally compact type space. In this paper we study the definability of the locally compact group image of the quasi-homomorphism in this second proof. We show that it is isomorphic as a topological group to a relatively hyperdefinable locally compact group. 
\end{abstract}

\vspace{0.50cm}

\section{Introduction}
The aim of this paper is to study the definability properties of the Lie Model theorem in \cite[Theorem 3.25]{KP}. This theorem was first proved in \cite[Theorem 4.2]{HR2} by Hrushovski in full generality. \\
It deals with $k$-approximate groups. The notion was introduced by Tao in \cite{TA}. 

\begin{definition}
     A \textbf{$k$-approximate group} is a subset $X$ of a group $G$ which is symmetric (i.e. $X^{-1}=X$) and such that $X^2\subseteq EX$ where $E\subseteq G$ is a set of cardinality $k$.
\end{definition}

Here $X^2$ denotes the set $\{xx' \ | \ x,x'\in X\}$, $EX$ analogously denotes the set $\{ ex \ |  \ e \in E  \text{ and } x \in X\}$ and $X^{-1}=\{x^{-1} \ | \ x  \in X\}$.\\

Before the general theorem, Hrushovski proved the Lie Model theorem for pseudofinite groups in \cite{HR}. The theorem, simplifying, states that there is a subset $Y$ of $X^4$ such that finitely many translates of $Y$ cover $X^4$ and there is a homomorphism of groups from $\langle X\rangle$ to a connected Lie group with kernel in $Y$. This led to the asymptotic classification of all finite approximate groups by Breuillard, Green and Tao in \cite{BGT}. \\
Hrushovski's proof is model-theoretic and requires an expansion of the language. Massicot and Wagner showed local hyperdefinability of the Lie group in the original language in \cite{MW}.\\ 

For the generalized version of the theorem for arbitrary infinite approximate groups, the homomorphism is replaced by a quasi-homomorphism. We recall the definition:

\begin{definition}
A \textbf{quasi-homomorphism} with error set $C$ is a function $f:G\rightarrow H$ where $G$ and $H$ are groups and $C\subseteq H$ is a set containing both $\{f(y)^{-1}f(x)^{-1}f(xy)\mid x, y \in G\}$ and $\{f(xy)f(y)^{-1}f(x)^{-1} \mid x, y \in G\}$ which is small in some sense. We write $f:G\rightarrow H:C$.  
\end{definition}

The general version of the Lie model theorem was first proved by Hrushovski in \cite[Theorem 4.2]{HR2}.

\begin{teo}
    For every approximate subgroup $X$ there is a quasi-homomorphism $f: \langle X\rangle \rightarrow H:C$ with $H$ a second countable locally compact topological group and $C$ a compact normal error set such that:
\begin{enumerate}
        \item For $K\subseteq H$ compact $f^{-1}(K)$ is contained in some $X^{n}$,
        \item For each $n\in\mathbb{N}$ there exists a compact $K\subseteq H$ with $X^n\subseteq f^{-1}(K)$,
        \item Every two compact subsets $K$, $K'$ of $H$ with $C^2K\cap C^2K'=\emptyset$ have preimage separated by a definable set.
    \end{enumerate} 
    Moreover, $f^{-1}(C)\subseteq X^{12}$.
\end{teo}

The theorem has connections with the theory of approximate lattices, in the sense that information about a given approximate lattice can be drawn from the Lie model; see, for example, \cite{SM}.\\
An alternative proof of the general Lie model theorem was given in \cite{KP} by Krupinski and Pillay. Their proof uses a generalization of topological dynamics to a locally compact space of types. In this paper, we prove that the image of the quasi-homomorphism defined in their article is homeomorphic as a topological group to a relatively hyperdefinable locally compact group. \\

\section{The Lie Model Theorem} \label{sec 2}
In this section we shall present the Lie Model theorem as stated in \cite[Theorem 3.25]{KP}. First, we recall the definition of a generalized definable locally compact model.

\begin{definition}
    Let $X$ be an approximate group, $G=\langle X\rangle$. A \textbf{generalized definable locally compact model} of $X$ is a quasi-homomorphism  $f:G\rightarrow H:C$ where $H$ is a Hausdorff locally compact group and $C$ is a normal symmetric compact subset of $H$ such that:
    \begin{enumerate}
        \item for every compact $K\subseteq H$ there is $n\in\mathbb{N}$ with $f^{-1}(K)\subseteq X^n$
        \item for every $n\in\mathbb{N}$ the closure $\overline{f(X^n)}$ is compact in $H$
        \item there is an $l\in\mathbb{N}$ such that for any compact $Z,Y\subseteq H$ with $C^lY\cap C^lZ=\emptyset$ the preimages $f^{-1}(Y)$ and $f^{-1}(Z)$ can be separated by a definable set. 
    \end{enumerate}
\end{definition}

We now fix the context in which we are working. Let $M$ be a model in a language $\mathcal{L}$ and $X$ an approximate subgroup definable in $M$. Let $N\succcurlyeq M $ be an $|M|^{+}$-saturated elementary extension and $\mathcal{U}\succcurlyeq N$ the monster model. Define $G:=\langle X\rangle$ the group generated by $X$. We denote by $\bar X$ the interpretation of $X$ in $\mathcal{U}$ and by $\bar G$ the set $\langle\bar X\rangle$. We are working in the type space $S_{G,M}(N)=\bigcup_n S_{X^n,M}(N)$ where for every $n\in\mathds{N}$ the set $S_{X^n,M}(N)$ is the set of complete types over $N$ finitely satisfiable in $M$ concentrating in $X^n$, and $S_M(N)=\bigcap_n\{[\neg\varphi] \ | \varphi \text{ not satisfiable in $M$ }\} $ is closed in $S(N)$.\\

We endow $S_{G,M}(N)$ and all $S_{X^n,M}(N)$ with the topology induced from $S(N)$, or equivalently from $S_M(N)$. It follows that all $S_{X^n,M}(N)$ are compact, so $S_{G,M}(N)$ is locally compact. Moreover, if a set $F\subseteq S_{G,M}(N)$ is closed, then $F\cap S_{X^n.M}(N)$ is closed for all $n$; conversely, if $F\cap S_{X^n,M}(N)$ is closed for each $n$, say $F\cap S_{X^n,M}(N)=[\pi_n]\cap S_{X^n,M}(N)$ for some partial type $\pi_n$, then $F\cap S_{G,M}(N)=[\varphi(x)\lor x\not\in X^n: \varphi\in\pi_n, n\in\mathbb{N}]$. Hence the globally closed sets of $S_{G,M}(N)$ coincide with the locally closed sets, i.e. intersections of a closed set with each $S_{X^n,M}(N)$.\\

The space $S_{G,M}(N)$ can be endowed with a semigroup structure. In fact, take $p,q\in S_{G,M}(N)$. Consider the product $p*q=\tp(ab/N)$ where $a\models p$, $b\models q$ and $\tp(a/N,b)$ is a coheir over $M$. Note that this is unique by $|M|^+$-saturation of $N$. This operation is associative and left continuous, so $(S_{G,M}(N), *)$ is a left topological semigroup \cite[Lemma 3.5]{KP}.\\

\subsection{Topological dynamics}

Krupinski and Pillay's proof uses an extension of topological dynamics to build the quasi-homomorphism of the theorem. Classical topological dynamics studies flows.

\begin{definition}
    A \textbf{flow} is a pair $(G, Y)$ where $Y$ is a compact Hausdorff space and $G$ is a topological group acting continuously on $Y$
\end{definition}

We are considering the pair $(G, S_{G,M}(N))$, where the action of $G$ on $S_{G,M}(N)$ is defined as $g\tp(a/N)= \tp(ga/N)$ for $g\in G$ and $\tp(a/N)\in S_{G,M}(N)$. It is well defined and it is an action by homeomorphism. Notice that $G$ embeds in $S_{G,M}(N)$ via $g\mapsto \tp(g/N)$ so, identifying $g\in G$ with $\tp(g/N)$, the operation $*$ we previously defined extends the action of $G$ to all $S_{G,M}(N)$. Moreover, $G$ is dense in $S_{G,M}(N)$. The problem is that $S_{G,M}(N)$ is not a compact space. \\ 

For this reason,  Pillay and Krupinski extend classical topological dynamics to the locally compact space of types $S_{G,M}(N)$. They show that $S_{G,M}(N)$ has a minimal left ideal $\mathcal{M}$ \cite[Proposition 3.10]{KP} which contains an idempotent $u$ \cite[Lemma 3.12]{KP} and $u*\mathcal{M}$ is a group, the Ellis group \cite[Lemma 3.13]{KP}. As proven in \cite[Lemma 3.11]{KP}, left continuity essentially implies that $\mathcal{M}=S_{G,M}(N)*c$ with $c\in\mathcal{M}\cap S_{X,M}(N)$ is relatively closed in $S_{G,M}(N)$, but it does not necessarily hold that $u*\mathcal{M}$ is relatively closed.\\
As in classical topological dynamics, they introduce the $\tau$-topology on the group $u*\mathcal{M}$ which is coarser than the topology of $u*\mathcal{M}$ induced from $S_{G,M}(N)$. It is built from an operation $\circ$ defined in \cite[Definition 3.16]{KP}.

\begin{definition}
    For any $p\in S_{G,M}(N)$ and $Q\subseteq S_{G,M}(N)$ we define $p\circ Q$ as the set of all $r\in S_{G.M}(N)$ such that $r=\lim_ig_i*q_i$ where $(g_i)_i$ is a net in $G$ and $(q_i)_i$ is a net in $Q$ and $\lim_ig_i=p$.
\end{definition}

This operation allows us to define a closure operator \cite[Lemma 3.17]{KP} that defines the closed sets of the $\tau$-topology.

\begin{definition}
  For a subset $Q\subseteq u*\mathcal{M}$ put $cl_{\tau}(Q)=(u*\mathcal{M})\cap (u\circ Q)$. This is a closure operator.
\end{definition}

Then we can define the set $H(u*\mathcal{M})=\bigcap_V cl_{\tau}(V)$ where $V$ is ranging over all $\tau$-neighborhoods of $u$. It is a normal subgroup of $u*\mathcal{M}$ and the space $u*\mathcal{M}/H(u*\mathcal{M})$ endowed with the quotient topology of $\tau$ is a locally compact Hausdorff group \cite[Proposition 3.24]{KP}. \\

We have all the elements to define the quasi-homomorphism which will be the generalized locally compact model of $X$ in the theorem: $f:G\rightarrow u*\mathcal{M}/H(u*\mathcal{M})$ with $g\mapsto (u*g*u)/H(u*\mathcal{M})$. The error set is the following set $C$ as built in \cite[Section 3.2]{KP}:

\begin{align*}
    &F_n:=\{x_1*y_1^{-1}*\dots *x_n*y_n^{-1} \ | \ x_i, y_i \in \bar G \text{ and } x_i\equiv_M y_i \text{ for all } i\leq n\}\\
    &\tilde{F}_n:=\{\tp(a/N)\in S_{G,M}(N) \ | \ a\in F_n\}\\
    & \tilde{F}:= (\tilde{F}_7\cap u*\mathcal{M}/H(u*\mathcal{M}))^{u*\mathcal{M}/H(u*\mathcal{M})}\\
    & \ \ \  =\{y*x*y^{-1} \ | \ x\in \tilde{F}_7\cap u*\mathcal{M}/H(u*\mathcal{M}) \text{ and } y\in u*\mathcal{M}/H(u*\mathcal{M})\}\\
    &C:= cl_{\tau}(\tilde{F})\cup cl_{\tau}(\tilde{F})^{-1}
\end{align*}

The error set $C$ is small in the sense that it is contained in $(\tilde{F}_{10}\cap u*\mathcal{M})/H(u*\mathcal{M})$ \cite[Lemma 3.29]{KP}.\\
We are now ready to state the general Lie Model theorem.

\begin{teo}{\cite[Theorem 3.25]{KP}}
    The function $f$ is a generalized definable locally compact model of $X$ with compact normal symmetric error set $C$. Moreover, $f^{-1}(C)\subseteq X^{30}$ and there is a compact neighborhood $U$ of the identity in $u*\mathcal{M}/H(u*\mathcal{M})$ such that $f^{-1}(U)\subseteq X^{14}$ and $f^{-1}(UC)\subseteq X^{34}$.
\end{teo}

\section{Definability of $u*\mathcal{M}/{\sim}$}

In this section, we build the relatively hyperdefinable group homeomorphic to the image. We use the following result in \cite[Proposition A.41]{RZ}. We denote by $\overline{u*\mathcal{M}}$ the closure in the induced topology on $S_{G,M}(N)$, so $\overline{u*\mathcal{M}}$ is relatively type-definable. From now on, we write $H$ for $H(u*\mathcal{M})$.

\begin{prop}\label{function}
    Denote by $\zeta:\overline{u*\mathcal{M}}\rightarrow u*\mathcal{M}$ the function that maps $m\mapsto u*m$ and by $\xi:u*\mathcal{M}\rightarrow u*\mathcal{M}/H$ the canonical projection to the quotient. Then the map $\eta=\xi\circ\zeta:\overline{u*\mathcal{M}}\rightarrow u*\mathcal{M}/H$ is continuous.   
\end{prop}

We show that the image $u*\mathcal{M}/H$ is isomorphic as a topological group to some quotient $\overline{u*\mathcal{M}}/{\sim}$, which is relatively hyperdefinable. As they are topologically isomorphic, $\overline{u*\mathcal{M}}/ {\sim}$ is also locally compact, so it still contains a subgroup $L$ of finite index and a normal subgroup $N$ such that $L/N$ is a Lie group. We use the map $\eta$ of Proposition \ref{function}. As $\overline{u*\mathcal{M}}$ is closed in the relative topology of $S_{G,M}(N)$, there is a partial type $p$ such that $\overline{u*\mathcal{M}}=p(x)\cap S_{G,M}(N)$. We now define the quotient $\overline{u*\mathcal{M}}/{\sim}$.

\begin{definition}
    Take $a$, $b\in\overline{u*\mathcal{M}}$. Then $a\sim b$ if and only if $\eta(a)=\eta(b)$.
\end{definition}

\begin{teo}
    The group $\overline{u*\mathcal{M}}/{\sim}$ is a relatively hyperdefinable group.
\end{teo}

\begin{proof}
  The quotient $u*\mathcal{M}/H$ is a Hausdorff space. Thus, for every $a$, $b$ and $c\in u*\mathcal{M}/H$  such that $a*b\neq c$ we can find in the quotient topology of $\tau$ an open neighborhood $O\ni a*b$ and an open neighborhood $U\ni c$ such that $O\cap U=\emptyset$. 
The quotient $u*\mathcal{M}/H$ is also a topological group, so the product $*$ is continuous. Consequently, we can find two open neighborhoods $V\ni a$ and $W\ni b$ such that $V*W\subseteq O$. \\
We now consider the preimage under the map $\eta$ of $U$, $V$ and $W$. Denote them as $\overline{U}=\eta^{-1}(U)$, $\overline{V}=\eta^{-1}(V)$ and $\overline{W}=\eta^{-1}(W)$. By continuity of the map $\eta$, the sets $\overline{U}$, $\overline{V}$ and $\overline{W}$ are open sets of the type topology restricted to $\overline{u*\mathcal{M}}$. Hence, we can find the following unions of formulas: $\overline{U}=\bigcup_{i\in I}\theta_{U_{abc}}^i\cap \overline{u*\mathcal{M}}$, $\overline{V}=\bigcup_{j\in J}\varphi_{V_{abc}}^j\cap\overline{u*\mathcal{M}}$ and $\overline{W}=\bigcup_{l\in L}\psi^l_{W_{abc}}\cap \overline{u*\mathcal{M}}$. We obtain that for every $i\in I$, $j\in J$ and $l\in L$, $\varphi^j_{V_{abc}}\cap\overline{u*\mathcal{M}}\subseteq\overline{V}$, $\psi^l_{W_{abc}}\cap\overline{u*\mathcal{M}}\subseteq \overline W$ and $\theta^i_{U_{abc}}\cap\overline{u*\mathcal{M}}\subseteq \overline{U}$. For every such triplet of elements and for every triplet of formulas satisfying this property, define the formula $\neg\varphi^j_{V_{abc}}(x)\lor\neg\psi^l_{W_{abc}}(y)\lor\neg\theta^i_{U_{abc}}(z)$. Then, we can define the following partial type $\pi(x,y,z)$:
\begin{equation*}
    \pi(x,y,z)=\Set{ \neg\varphi^j_{V_{abc}}(x)\lor\neg\psi^l_{W_{abc}}(y)\lor\neg\theta^i_{U_{abc}}(z)\ | \begin{array}{l}
    \ a,b,c \in u*\mathcal{M}/H(u*\mathcal{M})\\
    \ a*b\neq c \\
     \ U, V, W \text{as in the proof }\\
    \ i\in I, j\in J, l \in L    
  \end{array}}
\end{equation*}
\begin{claim}
    The type $\pi(x,y,z)$ is equivalent to $\eta(x)*\eta(y)=\eta(z)$ in $\overline{u*\mathcal{M}}$.
\end{claim}
\begin{proof}
    $(\Rightarrow)$ Suppose $\eta(a)*\eta(b)\neq\eta(c)$ for some $a$, $b$, $c\in\overline{u*\mathcal{M}}$. Denote by $\tilde{a}=\eta(a)$, $\tilde{b}=\eta(b)$ and $\tilde{c}=\eta(c)$. Then, as in the first part of the proof, we can find the open sets $U\ni \eta(c)$, $V\ni\eta(a)$ and $W\ni\eta(b)$ with respectively preimages $\overline{U}$, $\overline{V}$ and $\overline{W}$ under $\eta$. These sets are a union of formulas we respectively denote by $\bigcup_{i\in I}\theta_{U_{\tilde{a}\tilde{b}\tilde{c}}}^i$, $\bigcup_{j\in J}\varphi_{V_{\tilde{a}\tilde{b}\tilde{c}}}^j$ and $\bigcup_{l\in L}\psi^l_{W_{\tilde{a}\tilde{b}\tilde{c}}}$ intersected with $\overline{u*\mathcal{M}}$. Then there exist $i\in I$, $j\in J$ and $l\in L$ such that $\varphi^j_{V_{\tilde{a}\tilde{b}\tilde{c}}}(a)$, $\psi^l_{W_{\tilde{a}\tilde{b}\tilde{c}}}(b)$ and $\theta^i_{U_{\tilde{a}\tilde{b}\tilde{c}}}(c)$ hold, which implies $\varphi^j_{V_{\tilde{a}\tilde{b}\tilde{c}}}(a)\land\psi^l_{W_{\tilde{a}\tilde{b}\tilde{c}}}(b)\land\theta^i_{U_{\tilde{a}\tilde{b}\tilde{c}}}(c)$. As by construction $\neg\varphi^j_{V_{\tilde{a}\tilde{b}\tilde{c}}}\lor\neg\psi^l_{W_{\tilde{a}\tilde{b}\tilde{c}}}\lor\neg\theta^i_{U_{\tilde{a}\tilde{b}\tilde{c}}}\in\pi$ we obtain $(a,b,c)\not\models\pi(x,y,z)$.\\
    
    $(\Leftarrow)$: Suppose $a$, $b$, $c\in\overline{u*\mathcal{M}}$ such that $\not\models\pi(a,b,c)$. This means that we can find some open neighborhoods $U$, $V$ and $W$ and some $i\in I$, $j\in J$ and $l\in L$ as in the first part of the proof such that $\varphi^j_{V_{abc}}(a)\land\psi^l_{W_{abc}}(b)\land\theta^i_{U_{abc}}(c)$. Then $\eta(a)\in W$, $\eta(b)\in V$ and $\eta(c)\in U$ but $W*V\cap U=\emptyset$ so $\eta(a)*\eta(b)\neq\eta(c)$.  
\end{proof}

Then the equivalence relation $\sim$ is relatively type-definable on $\overline{u*\mathcal{M}}$ by $\pi(x,e,y)$ where $e\in \overline{u*\mathcal{M}}$ such that $\eta(e)=u* H$. In fact, $\eta(x)=\eta(y)\iff\eta(x)*u*H=\eta(y)\iff\eta(x)*\eta(e)=\eta(y)\iff \pi(x,e,y)$. \\

This means that on $\overline{u*\mathcal{M}}/{\sim}$ we can define a product $\bar*$ defined as $a/{\sim} \ \  \bar* \ \ b/{\sim}=c/{\sim}$ if and only if $\eta(a)*\eta(b)=\eta(c)$. It is well defined and $a/{\sim} \ \ \bar* \ \ b/{\sim}=a*b/{\sim}$. Indeed $a/{\sim} \ \ \bar*\ \ b/{\sim}=a*b/{\sim}$ if and only if $\eta(a)*\eta(b)=\eta(a*b)$ but $\eta(a)*\eta(b)=(u*a/H)*(u*b/H)=(u*a*u*b)/H=(u*a*b)/H=\eta(a*b)$. The product is relatively hyperdefinable by the type $\pi$. \\
\end{proof}

We want to show that $u*\mathcal{M}/H$ is isomorphic as a topological group to $\overline{u*\mathcal{M}}/{\sim}$, which is relatively hyperdefinable. In order to do so we define a new topology on $\overline{u*\mathcal{M}}/{\sim}$.

\begin{definition}
    Consider the map $\bar\eta:\overline{u*\mathcal{M}}/{\sim}\rightarrow u*\mathcal{M}/H$ defined as $\bar\eta(x/{\sim})=\eta(x)$. It is well defined by definition of $\sim$. As $u*\mathcal{M}\subseteq \overline{u*\mathcal{M}}$, by continuity of $\eta$ and by definition of $\sim$, the map $\bar\eta$ is a continuous bijection. We define a topology $\tau_{\bar\eta}$ on $\overline{u*\mathcal{M}}/{\sim}$ whose open sets are the sets which are preimage of open sets of $u*\mathcal{M}/H$ under $\bar\eta$.
\end{definition}

\begin{teo}
    The group $\overline{u*\mathcal{M}}/{\sim}$ with the topology $\tau_{\bar\eta}$ is isomorphic to $u*\mathcal{M}/H(u*\mathcal{M})$ as a topological group.
\end{teo}

\begin{proof}
   By definition of $\tau_{\bar\eta}$, $\bar\eta$ is a topological homeomorphism. It preserves the product as $\bar\eta(a/{\sim} \ * \ b/{\sim})=\bar\eta(a*b/{\sim})=\eta(a*b)=\eta(a)*\eta(b)=\bar\eta(a/{\sim})*\bar\eta(b/{\sim})$. Thus, the map $\bar\eta$ is a homemomorphism of topological groups as we wanted. 
\end{proof}

\section{One-point compactification of $cl(u*\mathcal{M})$}

In the previous section, we worked on the space $\overline{u*\mathcal{M}}/{\sim}$, where $\overline{u*\mathcal{M}}$ is the relative closure of $u*\mathcal{M}$ in $S_{G,M}(N)$. This space is a locally compact Hausdorff space as it is homeomorphic to the space $u*\mathcal{M}/H$. We want to build the one-point compactification of this space. \\

Consider the closure of $u*\mathcal{M}$ in the type space $S(N)$. We denote this closure by $cl(u*\mathcal{M})$. Of course, by definition of relative topology, $\overline{u*\mathcal{M}}\subseteq cl(u*\mathcal{M})$. We denote by $Y$ the set $cl(u*\mathcal{M})\setminus \overline{u*\mathcal{M}}$. We want to extend the equivalence relation $\sim$ on the set $cl(u*\mathcal{M})$ so that there is a single equivalence class in $Y$. We also require that for every $x\in Y$ and $a\in \overline{u*\mathcal{M}}$ it holds that $x/{\sim} \ \tilde{*} \ a/{\sim}=x/{\sim}=a/{\sim} \ \tilde{*} \ x/{\sim}$ where $\tilde{*}$ is an operation that extends $\bar*$ to $cl(u*\mathcal{M})/{\sim}$.\\

To do this, we first redefine the type $\pi$ in a way such that it stays the same on $\overline{u*\mathcal{M}}$ and is always satisfied if we pick at least two elements in $Y$. We define the following type: 
\begin{equation*}
     \bar\pi(x,y,z)=\{\varphi(x,y,z)\lor x\not\in X^n\lor y\not\in X^n\lor z\not\in X^n \ | \ \varphi\in\pi \text{ and } n\in \mathbb{N}\}
\end{equation*}
In this way if one of $x$, $y$ or $z$ is in $Y$, $\bar\pi$ holds. The problem is that it also holds if we pick only one element in $Y$ and the other two in $\overline{u*\mathcal{M}}$. For this reason, we add the following formulas: 
\begin{equation*}
    \Set{\begin{array}{l}
    \ (x\in X^n\land y\in X^m)\rightarrow z\in X^{n+m}\\
    \ (x\in X^n\land z\in X^m)\rightarrow y\in X^{n+m}\\
    \ (y\in X^n\land z\in X^m)\rightarrow x\in X^{n+m}\\    
  \end{array} \ |\  n,m\in\mathbb{N}}\
\end{equation*}
At this point, as in the previous section, we define $x\sim y$ if and only if $\bar\pi(x, e, y)$. Clearly, if $x$, $y\in Y$, then it always holds. Instead, if one of $x$ or $y$ is in $\overline{u*\mathcal{M}}$ it does not hold.  Defining the product on $cl(u*\mathcal{M})/{\sim}$ as $x/{\sim} \ \tilde{*}  \ y/{\sim}= z/{\sim}$ if and only if $\bar\pi(x,y,z)$ we get that for $x\in Y$ and $a\in cl(u*\mathcal{M})$ then $x/{\sim} \ \tilde{*} \ a/{\sim}=x/{\sim}$ as we wanted. Notice that for this last property the product is well defined.  \\

In conclusion, we built a space $cl(u*\mathcal{M})/{\sim}=\overline{u*\mathcal{M}}/{\sim}\sqcup Y/{\sim}$ where $Y/{\sim}=\{x/{\sim}\}$ for $x\in Y$ is a singleton and for every $a/{\sim}\in cl(u*\mathcal{M})/{\sim}$ and $x\in Y$ their product is $x/{\sim}\ \tilde{*} \ a/{\sim}=x/{\sim}= a/ {\sim} \ \tilde{*} \ x/{\sim}$. Thus, $cl(u*\mathcal{M})/{\sim}$ is the one-point compactification of $\overline{u*\mathcal{M}}/{\sim}$. 
\

\

\end{document}